\documentclass[11pt]{amsart}
\usepackage{amsmath}
\usepackage{amsfonts}
\usepackage{amsthm}
\usepackage{amssymb}
\usepackage{mathrsfs}
\usepackage{latexsym}
\usepackage{verbatim}
\usepackage{diagrams}
\usepackage{a4wide}
\usepackage[mathcal]{euscript}
\usepackage{enumerate}
\usepackage{hyperref}
\usepackage{stmaryrd}

\def\sp{\mathrm{sp}}
\def\unr{\mathrm{unr}}

\def\ilim{\varprojlim}
\def\SL{\mathrm{SL}}
\def\JJ{\mathscr{J}}
\def\JJg{\JJ^{\mathrm{glob}}}
\def\rank{\mathrm{rank}}
\def\Sym{\mathrm{Sym}}
\def\co{C}
\def\be{B}

\def\mmd{\mathrm{mod}}
\def\RR{\widetilde{R}}

\def\Rboxone{R^{1,\square}}

\def\Rn{R^{\natural}}

\def\Rbox{R^{\square}}
\def\Rboxn{R^{\square,\natural}}
\def\Rmod{R^{\mathrm{mod}}}
\def\Rlocmod{R^{\mathrm{mod}}}
\def\Rloc{R^{\mathrm{loc}}}

\def\expp{\kern+0.1em{p}}

\def\red{\mathrm{red}}
\def\NQ{N\kern-0.05em{Q}}

\def\T{\mathbf{T}}
\def\Tan{\T^{\mathrm{an}}}

\def\TI{\widetilde{I}}
\def\Tw{\widetilde{\T}}
\def\Tone{\T^1}
\def\Tp{\T}

\def\Done{D^1}

\def\univ{\mathrm{univ}}
\def\Runiv{R^{\univ}}
\def\Rtwuniv{\widetilde{R}^{\univ}}

\def\Rtw{R^{\dagger}_{\ell}}

\def\wm{\widetilde{\m}}

\def\MaxSpec{\mathrm{MaxSpec}}

\def\length{\mathrm{length}}
\def\univ{\mathrm{univ}}
\def\CC{\mathcal{C}}
\def\BB{\mathcal{B}}
\def\AA{\mathcal{A}}

\def\Tor{\mathrm{Tor}}
\newarrow{Equals}{=}{=}{=}{=}{=}
\def\m{\mathfrak{m}}
\def\T{\mathbf{T}}
\def\Z{\mathbf{Z}}

\def\OL{\mathcal{O}}
\def\Q{\mathbf{Q}}

\def\QQ{T}

\def\F{\mathbf{F}}

\def\End{\mathrm{End}}
\def\Ann{\mathrm{Ann}}
\def\GL{\mathrm{GL}}

\def\rhobar{\overline{\rho}}
\def\ad{\mathrm{ad}}
\def\aa{\mathfrak{a}}

\def\eps{\epsilon}

\def\Frob{\mathrm{Frob}}
\def\Spec{\mathrm{Spec}}
\DeclareMathOperator\Tr{tr}
\def\tr{\mathrm{Tr}}

\newtheorem{theorem}{Theorem}[section]

\newtheorem{df}[theorem]{Definition}
\newtheorem{lemma}[theorem]{Lemma}

\newtheorem{prop}[theorem]{Proposition}
\newtheorem{corr}[theorem]{Corollary}

\newtheorem{remark}[theorem]{Remark}

\def\Rtd{\widetilde{R}^{\dagger}}

\begin{document}

\title{Non-minimal modularity lifting in weight one}
\author{Frank Calegari}
\thanks{The author was supported in part by NSF  Grant
  DMS-1404620.}

\begin{abstract} We prove an integral~$R =\T$ theorem for odd  two
dimensional~$p$-adic representations of~$G_{\Q}$ which are unramified at~$p$,
 extending results of~\cite{CG}
to the non-minimal case. We prove, for any~$p$, the existence 
of Katz modular forms modulo~$p$ of weight one  which do not lift to
characteristic zero.
\end{abstract}

\maketitle

\section{Introduction}

The main innovation of~\cite{CG} was to develop a framework for modularity
lifting theorems in contexts in which the Taylor--Wiles method did not apply. One
of the main examples in~\cite{CG} (Theorem~1.4) was a
 minimal modularity lifting theorem for odd two-dimensional Galois representations
which are unramified at $p$. This result was simultaneously a generalization  and a
specialization of the main theorem of Buzzard--Taylor~\cite{BuzzT};  generalized because it 
related  Galois representations  modulo~$\varpi^n$ to Katz modular forms  of weight one modulo~$\varpi^n$
neither of which  need lift to characteristic zero,
and specialized because it required a minimality hypothesis at primes away
from $p$. 
One of the goals of the present paper is to provide a theorem which is a new proof of many cases 
of~\cite{BuzzT}  in the spirit of~\cite{CG}. Our methods could be viewed as hybrid of both~\cite{CG}
and~\cite{BuzzT} in the following sense: as in~\cite{CG}, we prove an integral~$R=\T$ theorem for torsion
representations by working directly  in weight one, however, as a crucial input, we use
ordinary modularity lifting theorems in higher weight (as in~\cite{BuzzT}, although we only need to work
in weight~$p$) in order to show that the patched Hecke modules see every component of the generic
fibre of the global deformation ring.
In order to simplify some of our arguments, we do not
strive for maximal generality.  The assumption that the representations are unramified at~$p$, however, seems essential for the method (if one does not use base change), in contrast to~\cite{wild}.
Let us fix a prime $p > 2$ and a local field $[E:\Q_p] < \infty$ with ring of integers $\OL$ and
residue field~$k = \OL/\varpi$.

\begin{theorem}  \label{theorem:example}
Let $p > 2$, and let
$\rho: G_{\Q} \rightarrow \GL_2(\OL)$
be a continuous odd Galois representation ramified at finitely many primes and unramified at $p$.
Suppose that $\rhobar$ is absolutely irreducible. If $\rho$ is ramified at a prime $\ell$,
assume that 
$\rho|D_{\ell}$ is reducible.
Then $\rho$ is modular of weight one.
\end{theorem}

This result will be deduced from our main result, which is an integral
$R = \T$ theorem  which we now describe.
Let 
$$\rhobar: G_{\Q} \rightarrow \GL_2(k)$$ be a continuous absolutely irreducible
odd representation unramified at~$p$.
For each $\ell$, let $\eps$ denote the cyclotomic character. Let $\psi$ denote the Teichmuller lift
of $\det(\rhobar)$.
Let $N = S \cup P$ be a set of primes not containing~$p$ such
that  $\psi$ is ramified exactly at the primes contained in $P$ and unramified at primes contained $S$.
By abuse of notation, we also let~$N$ denote the product of the conductor of~$\psi$ with the primes in~$S$.
We consider the  functor~$D^1$ from complete local Noetherian
 $\OL$-algebras $(A,\m)$ with residue field $k$ defined (informally) as follows.
 Fix a collection of elements~$a_{\ell} \in k$ for~$\ell$ dividing~$N$. Let
 $D^1(A)$ consist of  deformations $\rho$ to $A$ together with a collection of elements $\alpha_{\ell} \in A$
 for $\ell \in N$ such that: 
\begin{enumerate}
\item $\det(\rho) = \psi$.
\item $\rho$ is unramified outside $N = S \cup P$.
\item If $\ell \in P$, then
$\rho |D_{\ell} \simeq 
 \chi^{-1}  \psi |_{D_{\ell}} \oplus \chi$ for some
 unramified character $\chi$ with $\chi(\Frob_{\ell}) = \alpha_{\ell}  \equiv  a_{\ell} \mod \m$.
\item If $\ell \in S$, then
  $\rho | D_{\ell} \simeq  \displaystyle{ \left( \begin{matrix} \chi^{-1} \psi|_{D_{\ell}} & * \\
0 & \chi \end{matrix} \right)}$
 for some
 unramified character $\chi$ with $\chi(\Frob_{\ell})  = \alpha_{\ell}
\equiv  a_{\ell} \mod \m$. 
\end{enumerate}

In fact, the actual definition of~$D^1(A)$ needs to be somewhat modified 
(see~\S\ref{section:definitions} for precise definitions),
but this description will be valid for rings of integers such as~$\OL$.
Naturally enough, we also assume  that $D^1(k)$ is non-empty, and
that $k$ is also
 large enough to contain
the eigenvalues of every element in the image of $\rhobar$.
The elements $a_{\ell} \in k$ are determined by $\rhobar$ for primes
in $P$, but not necessarily for primes in~$S$, because when~$\rhobar$
is unramified at~$\ell$, there is a choice of eigenvalue for the unramified line.
Hence $D^1$ will not strictly be a Galois deformation ring;
we refer to such rings (and we have several in this paper) as modified
deformation rings because they depend not only on~$\rhobar$ but also
on some auxiliary data.
The functor $D^1$ is representable by a complete local $\OL$-algebra $R^1$.
The ring $R^1$ comes with elements $a_{\ell} \in k$  and~$\alpha_{\ell} \in R^1$ for $\ell$ dividing~$N$. For all other
primes $\ell$, define $a_{\ell} \in k$ to be $\Tr(\rhobar(\Frob_{\ell}))$, including when
$\ell = p$.

We are now ready to state out main theorem. By abuse of notation, let $N$ denote the conductor of $\psi$ times
the primes in $S$ --- it is divisible exactly by the primes in $N = S \cup P$.
Let $X_H(N)$ denote the quotient of $X_1(N)$ by the Sylow $p$-subgroup of $(\Z/N\Z)^{\times}$. 
After enlarging~$S$ if necessary, we may assume that the curve~$X_H(N)$ is a fine moduli space.
($X_H(N)$ will be automatically a fine moduli space if~$p > 3$, see~\S\ref{section:stacks}.)

\begin{theorem} \label{theorem:main} Let $p > 2$.
Let $\T \subset \End_{\OL} H^0(X_H(N),\omega_{E/\OL})$ be the $\OL$-algebra
generated by Hecke endomorphisms.
 Let $\m$ be the maximal ideal of $\T$ generated by the elements
$\langle \ell \rangle - \psi(\ell)$ for $\ell \nmid N$ and 
$T_{\ell} - a_{\ell}$ for all $\ell$. Then there is an isomorphism
$R^1 \simeq \T_{\m}$.
\end{theorem}

\subsection{Theorem~\ref{theorem:main} implies Theorem~\ref{theorem:example}}

Suppose that $\rho: G_{\Q} \rightarrow \GL_2(\OL)$ is a continuous
 Galois representation ramified at $\ell \ne p$ satisfying the conditions of
 Theorem~\ref{theorem:example}. Then after a global twist and enlarging $\OL$ if necessary
 to contain a choice of elements $a_{\ell}$ for $\ell$ dividing $N$, $\rho$ gives rise to an
 element of $D^1(\OL)$.
 The modularity of $\rho$ then follows.
 
\medskip

As an application of Theorem~\ref{theorem:main}, we prove the following:

\begin{theorem} \label{theorem:blog} Let $p$ be any prime. There exists 
a Katz modular form $f \in H^0(X(\Gamma)_{\F_p},\omega)$ for some level $\Gamma$ prime to $p$ which
does not lift to characteristic zero. 
\end{theorem}

The original argument of Wiles~\cite{W,TW}  for modularity theorems at non-minimal level was to use an induction
argument  and a certain numerical criterion involving complete intersections which were
finite over $\OL$. This does not seem
to be obviously generalizable to weight one --- although one still has access to forms of Ihara's Lemma,
the Hecke rings $\Tone$ are no longer complete intersections in general, and are certainly
not flat over $\OL$. It remains open as to whether one can proceed using such an argument. 
Instead, we use modularity theorems in weight~$p$ in order to show $R^1_{Q}/\varpi \simeq
\Tone_{Q,\m}/\varpi$ for various sets of auxiliary primes~$Q$, and we then use this information to show
that the patched Hecke modules in weight one are ``big enough.'' To pass between weight one and
weight~$p$ we \emph{crucially} rely on $q$-expansions.
For this reason, the methods of this paper will probably not be generalizable beyond $\GL(2)$
(although they may have implications for Hilbert modular forms of partial weight one).
Note that, in writing the paper~\cite{CG}, we tried to avoid the use of $q$-expansions as much as possible, whereas
the philosophy of this paper is quite the opposite.

\begin{remark} \emph{The methods of our paper may well be able to handle
more precise local deformation conditions than those considered above. 
However,
these assumptions considerably simplify some aspects of the arguments. We particularly shun
Diamond's vexing primes, which did indeed cause considerable vexation in~\cite{CG}.  In fact,
we try so hard to avoid them that we assume that $\rhobar |D_{\ell}$ is reducible, when certainly
some such representations --- for example those with $\rhobar | I_{\ell}$ irreducible --- may well
be amenable to our methods.}
\end{remark}

\subsection{Acknowledgements}
The debt this paper owes to~\cite{CG} is clear, and the author thanks David Geraghty for many conversations. 
We thank
Mark Kisin for the explaining a proof of
Lemma~\ref{lemma:kisin}, and we also thank Brian Conrad for a related proof of
the same result
in the context of rigid analytic geometry.
We thank Toby Gee and Patrick Allen for several useful comments. 
We also thank Gabor Wiese for the original idea of proving modularity theorems in weight~one by
working in weight~$p$.

\section{Preliminaries}
\label{section:prelim}

\subsection{Local modified deformation rings}
Let~$\rhobar: G_{\Q_{\ell}} \rightarrow \GL_2(k)$ be a representation, and let~$\Runiv_{\ell}$
denote the universal framed local deformation ring, and~$\rho^{\univ}$ the universal local deformation.
We assume in this entire section that~$p > 2$. The calculations
in this section will mostly be concerned with the case that~$\ell \ne p$.
Fix a lift of Frobenius~$\phi \in G_{\Q_{\ell}}$, and choose an eigenvalue~$a_{\ell}$
of~$\rhobar(\phi)$, which, after enlarging~$k$ if necessary, we may assume to lie in~$k$.
We define the universal modified framed local deformation ring 
$\Rtwuniv_{\ell}$ to be the localization
of the ring
$$\Runiv_{\ell}[\alpha_{\ell}]/(\alpha^2_{\ell} -  \alpha_{\ell} \mathrm{Tr}(\rho^{\univ}(\phi))
+ \det(\rho^{\univ}(\phi)))$$
at~$(\alpha_{\ell} - a_{\ell})$. The quadratic polynomial satisfied by~$\alpha_{\ell}$
is the characteristic polynomial of Frobenius.
\begin{lemma} \label{lemma:hensel} If~$\rhobar(\phi)$ has distinct eigenvalues, then $\Rtwuniv_{\ell} \simeq \Runiv_{\ell}$.
If~$\rhobar(\phi)$ does not have distinct eigenvalues, then $\Rtwuniv_{\ell}$ is a finite flat
extension of $ \Runiv_{\ell}$ of degree two.
\end{lemma}

\begin{proof} If~$\rhobar(\phi)$ has distinct eigenvalues, then the characteristic polynomial
of Frobenius is separable over~$k$. Since~$\Runiv_{\ell}$ is complete, the
polynomial also splits over~$\Runiv_{\ell}$ by Hensel's Lemma, and the quadratic extension above is,
(before localization), isomorphic to $\Runiv_{\ell} \oplus \Runiv_{\ell}$. Localizing
at~$(\alpha_{\ell} - a_{\ell})$ picks out the factor on which we have the 
congruence~$\alpha_{\ell} \equiv a_{\ell} \mod \m$.
If the eigenvalues of~$\rhobar(\phi)$ are both~$a_{\ell}$, then the quadratic extension is already local.
\end{proof}

A modified local deformation ring will simply be a quotient of~$\Rtwuniv_{\ell}$.
Proposition~3.1.2 of~\cite{gee} proves the
existence of  quotients~$R^{\univ,\psi,\tau}_{\ell}$
of~$\Runiv_{\ell}$ 
which are reduced, $\OL$-flat, equidimension of dimension~$4$, and such that,
for any finite extension~$F/E$, a map
$$x: R^{\univ} \rightarrow F$$
factors through~$R^{\univ,\psi,\tau}$ if and only if the 
corresponding~$F$ representation~$V_x$
has determinant~$\psi$ and is of type~$\tau$. (For this section~$\psi$
may be any unramfied character.)
For our purposes, it will suffice to
consider the trivial type~$\tau$, which corresponds to representations on which
$$\rho_x: G_{\Q_{\ell}} \rightarrow \GL(V_x) = \GL_2(F)$$
restricted to the inertial subgroup~$I_{\ell} \subset G_{\Q_{\ell}}$ has unipotent (and so
possibly trivial) image.

\begin{lemma} \label{lemma:gee}
Suppose that~$p > 2$ and~$\ell \ne p$.
 Let~$\tau$ denote the trivial type. There exists a quotient
$R^{\mmd,\psi}_{\ell} := \RR^{\univ,\psi,\tau}_{\ell}$ of~$\Rtwuniv_{\ell}$
which is reduced, $\OL$-flat, equidimensional of dimension~$4$,
and such that, for any finite extension~$F/E$, a map
$$x: \Rtwuniv_{\ell} \rightarrow F$$
factors through~$R^{\mmd,\psi}_{\ell}$ if and only if the 
corresponding~$F$ representation~$V_x$ has determinant~$\psi$,
is ordinary, and has an unramified quotient on which the action of~$\Frob_{\ell}$
is by the image of~$\alpha_{\ell}$.
\end{lemma}

The arguments are very similar to those already in the literature, but for want
of a reference which covers this case exactly, we give the details.

\begin{proof} Suppose that the eigenvalues of~$\rhobar(\phi)$ are distinct.
Then, under the isomorphism~$\Runiv_{\ell} \simeq \Rtwuniv_{\ell}$,
 we may take~$ \RR^{\univ,\psi,\tau}_{\ell} = R^{\univ,\psi,\tau}_{\ell}$.
 Hence we may assume that the eigenvalues are the same.
   Any representation
 $$\rho_x: G_{\Q_{\ell}} \rightarrow \GL_2(F)$$
for which the image of inertia has non-trivial unipotent image is, up to twist, and
enlarging~$F$ if necessary, 
an extension of~$F$ by~$F(1)$. In particular, the ratios of the eigenvalues of~$\rho_x(\phi)$
must be equal to~$\ell$. Since we are assuming
the eigenvalues of~$\rhobar(\phi)$ coincide,
then, if~$\ell \not \equiv 1 \mod p$,~$ R^{\univ,\psi,\tau}_{\ell}$ has no such quotients, and 
will consist precisely the unramified locus. In this case,
 we may take $\RR^{\univ,\psi,\tau}_{\ell}$ to be the double cover corresponding
to the unramified locus with a choice of Frobenius eigenvalue. 
Hence we may assume that~$\ell \equiv 1 \mod p$, and in 
particular~$\ell \not\equiv -1 \mod p$.

Assume that~$\rhobar$ is unramified.
The ring~$R^{\univ,\psi,\tau}_{\ell}$
admits  two natural quotients; a quotient~$R^{\unr,\psi}_{\ell}$ corresponding to
representations which are  unramified, and a quotient
corresponding to representations for which the ratios of the eigenvalues of Frobenius are equal to~$\ell$.
Because the determinant is fixed, this latter quotient is given by imposing the the equation
$$\Tr(\rho(\phi))^2 =  \ell^{-1} (1 + \ell)^2 \psi(\ell).$$
Since~$\ell \not \equiv -1 \mod p$, the right hand side is a unit, and hence there is exactly
one square root of this equation which is compatible with the choice of~$a_{\ell}$, and so this is equivalent
to the equation
$$\Tr(\rho(\phi)) =  \ell^{-1/2} (1 + \ell) \psi^{1/2}(\ell)$$
for the appropriate choice of square root.
The  ring obtained by imposing this relation on~$R^{\univ,\psi,\tau}_{\ell}$ may
or may not be either~$\OL$-flat or reduced, but let~$R^{\sp,\psi}_{\ell}$ denote
the largest quotient with this property ($\sp$ is for special). Its~$F$-points will still  include all ramified
representations of type~$\tau$.
 The pre-image of the corresponding affine scheme under
the projection $\Spec(\Rtwuniv_{\ell}) \rightarrow \Spec(\Runiv_{\ell})$ is equal to
$$R^{\sp,\psi}_{\ell}[\alpha_{\ell}]/(\alpha^2_{\ell} - \ell^{-1/2} (1 + \ell) \psi^{1/2}(\ell) \cdot \alpha_{\ell} + \psi(\ell)).$$
The quadratic relation factors as
$$(\alpha_{\ell} - \psi^{1/2}(\ell) \cdot \ell^{1/2})(\alpha_{\ell} - \psi^{1/2}(\ell) \cdot \ell^{-1/2}).$$
Define~$\RR^{\sp,\psi}_{\ell}$ to be the quotient on which~$\alpha_{\ell} = \psi^{1/2}(\ell) \cdot \ell^{-1/2}$.
There is a corresponding isomorphism
$$R^{\sp,\psi}_{\ell} \rightarrow \RR^{\sp,\psi}_{\ell}.$$
On the other hand, the quotient $R^{\unr,\psi}_{\ell}$ is a formally smooth. In this case, we
let~$\RR^{\unr,\psi}_{\ell}$ be the finite flat degree two extension given by adjoining
an eigenvalue~$\alpha_{\ell}$ of the characteristic polynomial of Frobenius.
We now let~$R^{\mmd,\psi}_{\ell} = \RR^{\univ,\psi,\tau}$ be the image of~$\Rtwuniv_{\ell}$
under the map
$$\Rtwuniv_{\ell} \rightarrow \RR^{\unr,\psi}_{\ell} \oplus \RR^{\sp,\psi}_{\ell}.$$
It is~$\OL$-flat and reduced because both~$ \RR^{\unr,\psi}_{\ell}$ and~$ \RR^{\sp,\psi}_{\ell}$
have this property. Moreover, the~$F$-points for finite extensions~$F/E$
 correspond exactly
to either an unramified representation together with a choice of Frobenius, or a ramified ordinary
representation together with~$\alpha_{\ell}$ being sent to the action of Frobenius on the unramified
quotient. Geometrically, $\Rmod_{\ell}$ consists of the union of two components, one the special
component of~$R^{\univ,\psi,\tau}_{\ell}$, and the other a double cover of the unramified
component of~$R^{\univ,\psi,\tau}_{\ell}$.
We also remark that, by construction, the image
of the universal deformation ring~$R^{\univ,\psi}$ in~$\Rmod_{\ell}$ will
be precisely~$R^{\univ,\psi,\tau}$.
\end{proof}

We also note the following:

\begin{corr} \label{corr:non} 
Suppose that~$p > 2$ and~$\ell \ne p$.
Let $x: \Rtwuniv_{\ell} \rightarrow E[\eps]/\eps^2$ be a surjective map so that the image
of~$\Runiv$ is~$E$, and the corresponding Galois representation~$\rho_x: G_{\Q_{\ell}}
\rightarrow \GL_2(E)$ is unramified. Then~$x$ factors through~$R^{\mmd,\psi}_{\ell}$.
\end{corr}

\begin{proof} Such a representation exists exactly when the eigenvalues of $\rho_x(\phi)$ are equal.
We see that~$x$ certainly factors through~$\RR^{\unr,\psi}$, which is a quotient 
of~$R^{\mmd,\psi}_{\ell}$.
\end{proof}

If the determinant is explicit from the context, we write~$\Rmod_{\ell}$ rather than~$R^{\mmd,\psi}_{\ell}$.
We give a precise description of the special fibre of~$\Rmod_{\ell}$ when~$\ell \equiv 1 \mod p$,
~$\psi = 1$, 
and~$\rhobar$ is trivial (this result will only be used for the proof
of Theorem~\ref{theorem:blog} in section~\S\ref{section:blog}.) Note that,
since~$\rhobar$ by assumption is tamely ramified at~$\ell$,
the image
of any deformation also factors through tame inertia, hence through the 
group $\langle  \tau,\phi \rangle$ with $\phi \tau \phi^{-1} = \tau^{\ell}$.

\begin{lemma} \label{lemma:special} Suppose that~$p > 2$, that~$\ell \equiv 1 \mod p$, that~$\psi | G_{\Q_{\ell}} \rightarrow
\GL_2(k)$ is trivial. Then~$\Rmod_{\ell}$ represents the functor of deformations of~$\rhobar$
to~$A$ together with an~$\alpha_{\ell} \in A$ satisfying the following conditions:
 \begin{itemize}
\item $\tr(\rho(\tau)) = 2$,
\item $(\rho(\tau) - 1)^2 = 0$,
\item $(\rho(\tau) - 1)(\rho(\phi) - \alpha_{\ell}) = 0$,
\item $(\rho(\phi) -   
\alpha^{-1}_{\ell})(\rho(\tau) - 1) = 0$,
\item $(\rho(\phi) -  \alpha_{\ell})(\rho(\phi) - \alpha^{-1}_{\ell}
) = 0$.
\end{itemize}
\end{lemma}

The argument is similar (but easier) to the corresponding
arguments of Snowden (\cite{Snowden}~\S4.5). In fact, our argument amounts
to the case~$d = 0$ of a theorem proved by Snowden for all integers~$d > 0$.
The only reason that Snowden does not consider this case is that, in his context, 
$d$ is the degree of a finite extension of~$\Q_p$.

\begin{proof} 
The  last equation says that~$\alpha_{\ell}$
satisfies the characteristic polynomial of Frobenius. Hence the functor is certainly represented
by a quotient~$\Rtw$ of the universal such ring~$\Rtwuniv_{\ell}$. Let us show that~$\MaxSpec(\Rtw) = \MaxSpec(\Rmod_{\ell})$
(inside~$\MaxSpec(\Rtwuniv_{\ell})$).
Let~$x: \Rtw \rightarrow F$ be a point of~$\MaxSpec(\Rtw)$. 
If~$\rho(\tau)$ is trivial, then the equations reduce to the statement that~$\alpha_{\ell}$
is an eigenvalue of Frobenius, and these correspond exactly to the unramified points of~$\Rmod_{\ell}$.
If~$\rho(\tau)$ is non-trivial, then, from the first equation, its minimal polynomial will be~$(X-1)^2$, and so, after conjugation,
has the shape
$$\rho(\tau) = \left( \begin{matrix} 1 & 1 \\ 0 & 1 \end{matrix} \right).$$
The other equations then imply that
$$\rho(\phi) = \left( \begin{matrix} \alpha^{-1}_{\ell} & * \\ 0 & \alpha_{\ell} \end{matrix} \right).$$
Finally, from the equation $\phi \tau \phi^{-1}  = \tau^{\ell}$, we deduce
that~$\alpha^{-2}_{\ell} = \ell$. In particular, the representation is, up to twist,
an extension of~$F$ by~$F(1)$, which corresponds exactly to points on the
special component of~$\Rmod_{\ell}$. It follows from
Corollary~2.3 of~\cite{KhareW2} (see also
Lemma~\ref{lemma:kisin} below) that~$\OL$-flat
reduced quotients of~$\Runiv_{\ell}$ are characterized by their~$F$ points
for finite extensions~$F/E$. Since we have shown that~$\Rtw$ and~$\Rmod_{\ell}$ have
the same such quotients, and since~$\Rmod_{\ell}$
is~$\OL$-flat and reduced, it suffices to show that~$\Rtw$ is~$\OL$-flat and reduced.
The special fibre~$\Rtw/\varpi$ 
 is exactly the completion
of $\CC_0$ at $c = (1;1;0)$  in the notation of~\cite{Snowden} (\S3.5).
The proof of this  is identical to the proof of Lemma~4.7.4 of~\cite{Snowden}.
On the special fibre, the equation~$(\rho(\tau) - 1)^2 = 0$ implies that~$(\rho(\tau) - 1)^p = 0$
and so~$\rho(\tau^p)$ is trivial, and~$\rho(\tau^{\ell}) = \rho(\tau)$. Hence
the action of conjugation by~$\rho(\phi)$ on~$\rho(\tau)$ is trivial.
In~\cite{Snowden}, the image of inertia factors through an exponent~$p$
commutative group which,
as a module for~$\F_p \llbracket T \rrbracket$ where~$1+T$ acts as conjugation by~$\sigma$,
is isomorphic to~$U = \F_p \oplus \F_p\llbracket T \rrbracket^{\oplus d}$. In our context, the action of inertia
commutes with~$\sigma$ and
factors through a group~$U = \F_p$.
In particular, letting~$m = \rho(\tau) - 1$ and~$\varphi$ be the image of~$\rho(\phi)$,
the tuple~$(\varphi,\alpha,m)$ is the corresponding point on~$\CC_0$.  The rest
of the argument follows the proof of Theorem~4.7.1 of~\cite{Snowden}.
The ring~$\CC_{0}$ has two
minimal primes (corresponding to $(\alpha - 1)$ and~$m$, which
come from the components~$\AA_2$ and~$\BB_0$ respectively, in the
notation of~\cite{Snowden}). On the other hand, as we have shown,
$\Rtw[1/\varpi]_{\red} = \Rmod[1/\varpi]$ has two minimal primes
corresponding to the unramified and ordinary locus, so~$\Rtw[1/\varpi]$
has two minimal primes, and so, by Propositions~2.2.1 and~2.3.1 of~\cite{Snowden},
it follows that~$\Rtw$ is~$\OL$-flat and reduced, and we are done.
\end{proof}

\subsection{The functors~$D_Q$ and~$D^1_Q$}
\label{section:definitions}

In this section,~$\rhobar$ will be a global Galois representation unramified at~$p$ with the primes~$N = S \cup P$
as in the introduction.
We now define modified deformation rings $R^1_{Q}$ and $R_{Q}$
for certain sets $Q$ of auxiliary primes distinct from $N$ and $p$.  
Let~$D_{\ell}$ denote  the decomposition group~$G_{\Q_{\ell}} \subset G_{\Q}$.
The superscript~$1$ refers to weight one,
and the lack of superscript will refer to weight~$p$.
Note that $R^1_{\emptyset} = R^1$.
Besides the representation~$\rhobar$, part of the data required
to define~$\Done_Q$ and~$D_{Q}$ consists of a
fixed choice of elements~$a_{\ell} \in k$ for~$\ell$ dividing~$N$ and~$Q$.
 Moreover,
for~$D_{Q}$, we also fix an~$a_p \in k$.
Let $\Done_Q(A)$ and $D_Q(A)$  consist of  deformations $\rho$ to $A$ and a collection of elements $\alpha_{\ell} \in A$
 for $\ell \in N$ (and $\alpha_p \in A$ for $D_Q(A)$) such that:
\begin{enumerate}
\item $\det(\rho) = \psi$, where $\psi$ is the Teichmuller
lift of~$\det(\rhobar)$ for $\rho \in \Done_{Q}(A)$,
and $\det(\rho) = \psi \eps^{p-1}$ for $\rho \in D_{Q}(A)$.
\item $\rho$ is unramified outside $N \cup Q = S \cup P \cup Q$ for $\rho \in \Done_{Q}(A)$, and unramified outside $N \cup S
\cup Q \cup \{p\}$ for $\rho \in D_{Q}(A)$.
\item \label{item:P} If $\ell \in P$, then
$\rho |D_{\ell} \simeq 
 \chi^{-1} \psi |_{D_{\ell}} \oplus \chi$ if $\rho \in \Done_{Q}(A)$ and $\rho |D_{\ell} \simeq 
 \chi^{-1} \psi \eps^{p-1} |_{D_{\ell}} \oplus \chi$  if $\rho \in D_Q(A)$, where
 $\chi$ is an  unramified character  and 
 $\chi(\Frob_{\ell}) = \alpha_{\ell}  \equiv  a_{\ell} \mod \m$.
\item \label{item:S} 
If $\ell \in S$, then
$\rho | D_{\ell}$ corresponds to an~$A$-valued quotient of~$R^{\mmd}_{\ell}$,
where we take the determinant to be~$\psi$ if~$\rho \in \Done_{Q}(A)$
and~$\psi \cdot \eps^{p-1}$   if~$\rho \in D_{Q}(A)$, and, in either case~$\alpha_{\ell}
 \in \Rmod_{\ell}$ is $a_{\ell} \mod \m$.
\item  \label{item:Q} If $\ell \in Q$, then  $\ell \equiv 1 \mod p$, and $\rhobar(\Frob_{\ell})$ has distinct eigenvalues. 
Then $\rho |D_{\ell} \simeq \phi^{-1} \psi |_{D_{\ell}}  \oplus \phi$, where $\phi$ is a character  of
$\Q^{\times}_{\ell} \subset G^{\mathrm{ab}}_{\Q_{\ell}}$ such that
$\phi(\ell) = \alpha_{\ell} \equiv a_{\ell} \mod \m$. 
\item If $\rho \in D_Q(A)$ and $\ell = p$, then $\rho|D_p$ is ordinary with eigenvalue $\alpha_p \equiv a_p \mod \m$.
\end{enumerate}

In order for these functors to be
non-zero,  the~$a_{\ell}$  for
$\ell \in N \cup Q$ must be chosen to be one of the eigenvalues of $\rhobar(\Frob_{\ell})$, and $a_p$ 
must be one of the eigenvalues
of $\rhobar(\Frob_p)$.  As always, we may extend scalars from $k$ to a field which contains
all necessary eigenvalues. 
For each $\ell \in N$, there exists a corresponding  universal framed local deformation
ring associated to our deformation problem. There is no subtlety in defining
these rings outside the case of primes in~$S$, and at the prime~$p$.
The first case was addressed in the previous section.
For~$\ell = p$,
we use the modified deformation rings as constructed by Snowden (\cite{Snowden}, see
in particular  \S4.6).
For each~$\ell$, we denote the corresponding modified local deformation ring (with
the appropriate determinant) by~$\Rmod_{\ell}$.

\begin{prop} \label{prop:CM} For all of the $\ell$ different from~$p$, the corresponding modified local deformation
ring  $\Rmod_{\ell}$ is an $\OL$-flat  reduced equidimensional  ring of relative dimension~$3$ over~$\OL$.  If $\ell = p$ and $D = D_{Q}$, then   $\Rmod_{\ell}$ is an $\OL$-flat  reduced
equidimensional ring of dimension~$4$ over~$\OL$. 
\end{prop}

\begin{proof} We consider each deformation ring in turn.
\begin{enumerate}
\item Suppose that $\ell \in P$. By assumption, $\psi$ is ramified at~$\ell$ and hence $\alpha_{\ell}$
is uniquely determined by $\rho |D_{\ell}$. Hence we recover the framed local deformation ring, and
the result follows from Lemma~4.11 of~\cite{CG}.
\item Suppose that $\ell \in S$. Then the result follows from Lemma~\ref{lemma:gee}.
\item Suppose that~$\ell \in Q$. The assumption that~$\ell \equiv 1 \mod p$ and 
that~$\rhobar(\Frob_{\ell})$ has no distinct eigenvalues implies that there is no
distinction between~$\Runiv_{\ell}$ and~$\Rtwuniv_{\ell}$. Moreover,
all deformations of~$\rhobar$ will be tamely ramified and split as a direct sum
of two characters, and so~$\Rmod_{\ell} = R^{\univ,\psi}_{\ell}$ in this case.
The ring~$R^{\univ,\psi}_{\ell}$ has the desired properties by a 
direct computation, see for example Proposition~7 of~\cite{Shotton}: it
may be identified with $\OL \llbracket X,Y,Z,P \rrbracket/((1+P)^{m} - 1)$, where~$m$ is
the largest power of~$\ell$ dividing~$p-1$.
\item If $\ell = p$, and $\rhobar(\Frob_p)$ has distinct eigenvalues, then the usual definition of 
an ordinary deformation ring $R_{p}$
requires a choice of eigenvalue of the unramified quotient, and hence $\Rmod_{p}$ is just the usual
Kisin ring $R_p$ in this case.
If $\rhobar(\Frob_p)$ has the same
eigenvalues, then the local modified deformation ring is exactly the completion
of $\mathcal{B}_1$ at $b = (1;1;0)$ considered in~\cite{Snowden}~\S3.4 and denoted by~$\Rtd$ 
in~\cite{CG}~\S3.7. The case when $\rhobar(\Frob_p)$
has the same eigenvalues but is non scalar corresponds to the localization of  $\mathcal{B}_1$
at $\left(\left( \begin{matrix}1 & 1 \\ 0 & 1 \end{matrix} \right);1;0\right)$. 
In either case, $\Rmod_{p}$ is~$\OL$-flat, reduced,  equidimensional of relative
dimension~$4$ (over~$\OL$), and Cohen--Macaulay. 
\end{enumerate}
\end{proof}

We also present here the following proposition which will be useful later.
 (cf. Lemma~3.4.12 of~\cite{kisin-moduli}.)

\begin{lemma} \label{lemma:kisin} Let~$A$ and~$B$ be complete local  Noetherian reduced~$\OL$-flat algebras with residue field~$k$.
Then~$A \widehat{\otimes}_{\OL} B$ is reduced and~$\OL$-flat.
\end{lemma}
 
\begin{proof}  The~$\OL$-flatness follows from Lemma~19.7.1.2
of section~$0$ of~\cite{EGAfourone}.
Because~$B$ is reduced, it follows from Corollary~2.3 of~\cite{KhareW2} that
the intersection of the kernels of all 
morphisms~$B \rightarrow \OL'$ for the ring
of integers finite extensions~$E'/E$ is trivial.
 Using this,  we may write $B$ as an inverse limit 
$B = \ilim B_i$, 
where each $B_i$ is reduced and finite flat over~$\OL$. Then~$C = \ilim A \otimes B_i$ (now we
can replace~$\widehat{\otimes}$ by~$\otimes$)  and it suffices to
prove the claim for the usual tensor product when~$B$ is finite flat over~$\OL$, which we now assume.
 Since~$C$ is~$\OL$-flat,
it suffices to show that~$C[1/\varpi] = A[1/\varpi] \otimes_E B[1/\varpi]$ is reduced. However,
this follows
from~\cite{bourbaki},  Chap.~V~\S15.5; Theorem~$3$(d)).
\end{proof}

\subsection{Modular Curves} \label{section:stacks}
Let  $N \ge 5$, and 
let $X_H(N) = X(\Gamma_H(N))$ denote the quotient of $X_1(N)$ by the Sylow $p$-subgroup of $(\Z/N\Z)^{\times}$
considered as a smooth
proper scheme over $\Spec(\OL)$~\cite{deligne-rapoport}. To be precise, the curve~$X_H(N)$
is a fine moduli space providing that either~$p \ge 5$ or~$p = 3$ and~$N$ divisible by
a prime~$q \ge 5$ such that $q \equiv -1 \mod 3$.
This follows  either from the computation of stabilizers at the CM points
(as in~\cite{MazurEisenstein},~\S2, p.64), or, in the second case, because~$X_H(N)$ is a cover of~$X_1(q)$.
If~$p = 3$ and~$X_H(N)$ is not a fine moduli space, 
 we simply add a prime~$q \ge 5$ and~$q \equiv -1 \mod 3$ to~$S$
such that~$\rhobar$ is unramified at~$q$.
If~$Q$ is a collection of auxiliary primes disjoint from~$N$, let~$X_H(\NQ)$ denote a quotient
of~$X_1(N)$ by the~$p$-Sylow subgroup of~$(\Z/N\Z)^{\times}$ 
and some subgroup of~$(\Z/Q \Z)^{\times}$.  (In practice, the cokernel of the corresponding
subgroup of~$(\Z/Q\Z)^{\times}$ will be a~$p$-group.)

\subsection{Hecke algebras}
Let $\omega$  be
 the usual pushforward $\pi_{*} \omega_{\mathcal{E}/X_H(N)}$ of the relative dualizing sheaf
along the universal generalized elliptic curve. If $A$ is an $\OL$-module, then 
let $\omega^n_A = \omega^n \otimes_{\OL} A$.
The (Katz) space of modular
forms of weight $k$ and level $N$ is defined to be 
$H^0(X_H(N),\omega^k_{A})$.

\medskip

We shall now consider a number of Hecke algebras, and discuss the relationship between them.
Our coefficient ring or module will either be  $A = \OL$, $A = E = \OL \otimes \Q$,  $A = \OL/\varpi = k$, $A = \OL/\varpi^n$,
or $A = E/\OL$ unless otherwise specified.

\begin{df} The  Hecke algebra $\T_A$ in weight~$k$ is the $A$ sub-algebra of 
$$\End_A(H^0(X_H(N),\omega^k_{A}))$$
generated by the operators $T_n$ for $n$ prime to $p$ and diamond operators $\langle d \rangle$
for $d$ prime to~$N$.
\end{df}

Note that this definition includes  the operators~$T_{\ell}$ for ${\ell}|N$. 
These operators can also be denoted by~$U_{\ell}$ (which is what we shall do below).
We now define a variant of these Hecke algebras where we include the Hecke operator at~$p$.
\begin{df} Let $\Tw_A$ denote the ring $\T_A$  together with the operator $T_p$.
\end{df}

Note that a maximal ideal $\m$ of $\T$ need no longer  a priori be maximal in $\Tw$.
That is,~$\Tw_{\m}$ will not always be a local ring. However, it will always be a semi-local ring,
that is, a direct sum of finitely many local rings.

\subsection{Hecke algebras at auxiliary level~$Q$}
Let $Q$ be a finite collection of primes congruent to $1 \mod p$ and distinct from $N$. 
Let $\T_{Q,A}$ denote the~$\OL$-algebra generated by Hecke operators away
from~$p$ 
acting
at level $X_H(\NQ)$ with coefficients in~$A$ together with diamond operators~$\langle d \rangle$
for $d$ prime to~$N$,
 and let~$\Tw_{Q,A}$ denote~$\T_{Q,A}$ together with the operator~$T_p$.

\medskip

Suppose that $\rhobar$ is a representation such  that the modified deformation ring $R$ is non-zero -- equivalently, that $D(k) \ne 0$.
Recall that~$N$ is equal to  the conductor of~$\psi$ times the primes in some auxiliary set~$S$ which includes (but may be
larger than)
the  set of primes~$\ell \nmid p$ where~$\rhobar | I_{\ell}$ is non-trivial and unipotent. In particular, if~$\ell$ divides~$S$,
then~$\ell$ divides~$N$ exactly once.
By Serre's Conjecture~\cite{KhareW}, any such $\rhobar$ is modular of level $N(\rhobar) | N$ and
weight~$p$, so we now specialize to the case of weight~$p$,   and let~$\T_{Q} = \T_{Q,\OL}$. 
Let $\wm$ be a maximal ideal of $\Tw_{Q}$  corresponding to~$\rhobar$ (and to a choice of~$a_{\ell}$
for all~$\ell$ dividing~$N$,~$Q$, and~$p$).
Let us also suppose that for every prime~$\ell$ dividing~$Q$, the matrix~$\rhobar(\Frob_{\ell})$
has distinct eigenvalues (since this is an assumption in part~\ref{item:Q} of the definition of~$D_Q$).

\begin{prop} \label{prop:action} 
There exists a deformation
$$\rho_Q: G_{\Q} \rightarrow \GL_2(\Tw_{Q,\wm})$$
of $\rhobar$ unramified outside $p\NQ$ such that~$\rho(\Frob_{\ell}) = T_{\ell}$ for~$\ell$ prime 
to~$p \NQ$. Let~$\rho'_Q  = \rho_{Q} \otimes \eta$,
where~$\eta^2 = \psi \cdot \eps^{p-1}  \cdot \det(\rho_Q)^{-1}$. Then~$\rho'_{Q}$ is a deformation
of~$\rhobar$ in~$D_Q(\Tw_{Q,\wm})$. In particular, there is a corresponding map
$$R_Q \rightarrow \Tw_{Q,\wm}$$
sending $\Tr(\rho^{\mathrm{univ}}(\Frob_{\ell})) \in R_Q$ to $\eta(\ell) \cdot T_{\ell}$
for $\ell$ not dividing $p\NQ$,  sending $\alpha_{\ell}$ to $\eta(\ell) \cdot U_{\ell}$ for $\ell$ dividing $NQ$, and sending 
$\alpha_p + p^{\expp-1} \psi(p) \alpha^{-1}_p$ to $\eta(p) \cdot T_p$, or equivalently, $\alpha_p$ to the unit root of
$$X^2 - \eta(p) \cdot  T_p X + \psi(p) p^{\expp-1} = 0,$$
which lies in $\Tw_{Q,\wm}$ by Hensel's Lemma.
\end{prop}

 This proposition is (mostly) an exercise in Atkin--Lehner--Li theory.
Indeed, if one assumes that the action of~$U_{\ell}$ on forms of level~$\ell \| N$ is semi-simple
(which conjecturally is always the case), then the space of modular forms under consideration
will decompose into a direct sum of eigenforms for all the Hecke operators in~$\Tw_{Q,\m}$,
and then the claim follows immediately from known local-global compatibility
for classical modular forms. (The only local--global compatibility we require is given
by Theorem~3.1 of~\cite{Fermat}.)
In practice, we have to allow for the possibility that~$U_{\ell}$
may not act semi-simply, although this is not difficult.

\begin{proof}
The space of modular forms of weight~$p$ is torsion free, so the Hecke algebra
is determined by its action on
$$H^0(X_H(\NQ),\omega^p)_{\wm} \otimes E.$$
It suffices to prove the proposition after further decomposing this space into a  direct sum of~$\Tw_{Q,\wm}$-modules.
Enlarging~$E$ if necessary, we may assume that all the eigenvalues of
all Hecke operators at level dividing~$\NQ$ are defined over~$E$. 
Let~$\Tan_Q$ denote the anaemic Hecke algebra consisting of 
endomorphisms of~$H^0(X_H(\NQ),\omega^p)$ generated by Hecke operators~$T_n$ for~$n$ prime
 to~$p\NQ$
and diamond operators~$\langle d \rangle$ for~$d$ prime to~$\NQ$.  There is a map~$\Tan_Q \rightarrow \Tw_Q$; 
let~$\m \subset \Tan_Q$ denote the inverse of the maximal ideal~$\wm$ 
(which is determined by~$\rhobar$).
Note that~$\m$ may correspond to several~$\wm$ in~$\Tw_{Q}$; the possible~$\wm$ are indexed by the possible
choices of~$a_{\ell}$ for~$\ell$ dividing~$p\NQ$. In any event, there will always be an inclusion:
$$H^0(X_H(\NQ),\omega^p)_{\wm} \otimes E \subset 
H^0(X_H(\NQ),\omega^p)_{\m} \otimes E.$$
(This would be an equality if we replaced the left hand side with a direct sum over all~$\wm$
which pull back to~$\m$.)
The space~$H^0(X_H(\NQ),\omega^p) \otimes E$ decomposes under~$\Tan_{Q}$ into eigenspaces indexed by 
newforms~$f$ of level dividing~$\NQ$. 
Associated to a cuspidal newform~$f$ is a Galois representation~$\rho_f$.
In particular, combining all these Galois representations over~$f$ with~$\rhobar_f = \rhobar$,
 we obtain a Galois
representation
$$\rho: G_{\Q} \rightarrow \GL_2(\Tan_{Q,\m} \otimes E).$$
Because the traces of Frobenius elements lie in~$\Tan_{Q,\m}$, and
because~$\rhobar$ is absolutely irreducible, we may take the image of this
Galois representation to land in~$\GL_2(\Tan_{Q,\m})$ by~$(2.6)$ of~\cite{Lenstra}.
To this point, we have simply reconstructed the usual construction of the Galois
representation into the (anaemic) Hecke algebra.

\medskip

Let~$\rho_Q$ denote the Galois representation induced by composing this
with the image of~$\Tan_{Q,\m}$ in~$\Tw_{Q,\wm}$. This will be the~$\rho_{Q}$
of the proposition. The reason for the twist by~$\eta$
is to match the determinant with the required determinant for the functor~$D_Q$.
The main point of this proposition is to show that
that the extra old forms associated to~$f$ (with their concomitant actions of~$U_{\ell}$)
contain exactly the extra information needed to obtain a  modified deformation of~$\rhobar$
of type~$D_Q$.
The eigenspaces  corresponding to~$f$ will  contribute to the localization at~$\m$ if
and only if~$\rhobar_{f} = \rhobar$.
In particular, the level of~$f$ must be divisible by the Serre conductor of~$\rhobar$, and hence the level of~$f$ is
of the form~$\NQ/D$, where~$D$ is only divisible by primes dividing either~$S$ or~$Q$.
Hence  the integer~$D$ is square-free and prime to~$\NQ/D$.
Suppose that~$D$ has~$d$ prime divisors.
The form~$f$  generates a space of~$2^d$ oldforms of level~$\NQ$
consisting of~$f = f(q)$ together with the forms~$f(q^m)$ for~$m|D$.
By Atkin--Lehner-Li theory (Theorem~9.4 of~\cite{Stein})  this exhausts the entire space of
oldforms associated to~$f$
which appear in~$H^0(X_H(\NQ),\omega^p)_{\m} \otimes \Q$. 
Let us now describe the action of~$U_{\ell}$ on these spaces for~$\ell$ dividing~$D$.
 Again by Atkin--Lehner, this is given as the
tensor product over~$\ell | D$ of a two-dimensional space on which~$U_{\ell}$ acts by the matrix
$$\left( \begin{matrix} \Tr(\rho_f(\Frob_{\ell})) & \langle \ell \rangle \ell^{p-1} \\
-1 & 0 \end{matrix} \right).$$
Here~$\Tr(\rho_f(\Frob_{\ell}))$ may also be identified with the eigenvalue of~$f$ under the Hecke operator~$T_{\ell}$ acting at level~$\NQ/D$.
The element~$ \Tr(\rho_f(\Frob_{\ell}))$ will lie in the image of~$\Tan_{Q}$ by the Cebotarev density theorem.
Note that the eigenvalues of this matrix are precisely
the eigenvalues of~$\rho_f(\Frob_{\ell})$.
There are now two possibilities:
\begin{enumerate}
\item The eigenvalues of~$\rho_f(\Frob_{\ell})$ are distinct.
In this case, the space of oldforms over~$E$ decomposes further into eigenspaces under~$U_{\ell}$. The eigenvalues of~$U_{\ell}$ will correspond
precisely to Galois representations together with a choice of eigenvalue of~$\rho_f(\Frob_{\ell})$.
Each choice of eigenvalue will contribute to the localization at~$\wm$ if and only if the corresponding eigenvalue
 is~$a_{\ell} \mod \varpi$.
 After the global twist to match determinants,
 such representations will naturally be algebras over~$\Rmod_{\ell}$, where~$\alpha_{\ell}$ is sent 
 to~$\eta(\ell) \cdot U_{\ell}$ (this follows by the construction of the rings~$\Rmod_{\ell}$,
 in particular Lemma~\ref{lemma:gee} for unramified primes of type~$S$.
\item The eigenvalues of~$\rho_f(\Frob_{\ell})$ are equal.  Call the unique eigenvalue~$b_{\ell}$.
From the explicit matrix description of the action of~$U_{\ell}$
above, we see that~$U_{\ell}$ is not a multiple of the scalar matrix, and so
it is not diagonalizable. In particular, in the Hecke algebra, the operator~$U_{\ell}$ satisfies the relation~$(U_{\ell} - b_{\ell})^2 = 0$.
However, once again (after twisting), there will be
 a map from~$\Rmod_{\ell}$ sending~$\alpha_{\ell}$ to~$\eta(\ell) 
\cdot U_{\ell}$, by Corollary~\ref{corr:non} (the eigenvalues can
only be the same for primes of type~$S$).
\end{enumerate}

We remark that the second case above conjecturally never occurs in weight~$\ge 2$
(see~\cite{Coleman}).

Let us now consider the operators~$U_{\ell}$ for~$\ell$ not dividing~$D$. In this case,
the Galois representation~$\rho_f$ is ramified at~$\ell$, and local--global compatibility of Galois representations
implies that, after twisting by~$\eta$, the Galois representation has an unramified quotient on which~$\Frob_{\ell}$ acts via~$U_{\ell}$,
and hence we have a natural map from~$\Rmod_{\ell}$ to~$\Tw_{Q,\wm}$ sending 
(after twisting)~$\alpha_{\ell}$ to~$\eta(\ell) \cdot U_{\ell}$. 
Finally, since (by definition)~$a_p \in k$ is a unit (it is an eigenvalue of an invertible matrix), the representation~$\rho_f$
is ordinary at~$p$, and action of Frobenius on the Galois
representations associated to any form~$f$ will  admit an unramified at~$p$ quotient on which~$\Frob_p$
acts as~$U_p$. Hence there will be a natural map from~$\Rmod_p$ to~$\Tw_{Q,\wm}$ sending~$\alpha_p$ to~$\eta(p) \cdot U_p$,
which is related to $T_p$ via the equation $U^2_p - T_p U_p + 
\langle p \rangle p^{\expp-1} = 0$. 
\end{proof}

\begin{remark} \emph{The theorem above is true in any weight~$k \ge 2$, providing that one modifies the
definition of~$D_Q$ to take into account the weight, and one still works in the ordinary context (so~$a_p \in k^{\times}$).}
\end{remark}

\subsection{Modularity lifting theorems in weight~\texorpdfstring{$p$}{p}}
\label{section:higher}

The main goal of this section is to prove the following:

\begin{theorem} \label{theorem:minimal} There is an isomorphism
$R_Q \rightarrow \Tw_{Q,\wm}$.
\end{theorem}

Before proving this theorem, we remark that the ``modularity'' theorem
one can deduce from this~$R = \T$ theorem is already well known. In particular,
one knows that~$(R_Q[1/\varpi])^{\mathrm{red}} =  (\Tw_{Q,\wm}[1/\varpi])^{\mathrm{red}}$
(our Hecke algebras will not be reduced if the action
of~$U_{\ell}$  is not semi-simple). Hence the content of this theorem
is to upgrade this known result to an integral statement. In order to see how one might do this,
note that the modifications of Taylor--Wiles due to Diamond, Kisin, and others  
(\cite{DiamondMult},\cite{kisin})
proceed
by constructing  a patched module~$M_{\infty}$ over a patched deformation ring~$R_{\infty}$
and a ring of auxiliary diamond operators~$S_{\infty}$.
By hook or by crook, one tries to prove that~$M_{\infty}$ is faithful (or nearly faithful) as 
an~$R_{\infty}$-module. To recover a classical statement, one takes the quotient
of~$R_{\infty}$ and~$M_{\infty}$ by
the augmentation ideal~$\aa$ of~$S_{\infty}$, and recovers the classical ring~$R$ and a module~$M$
of classical modular forms
on which~$R$ acts via the quotient~$\T$.
An essential difficulty, however, is that even if one knows that~$M_{\infty}$ is faithful as an~$R_{\infty}$-module, this does not imply that~$M_{\infty}/\aa$ is faithful as an~$R_{\infty}/\aa$-module; that is, faithfulness is 
not preserved under quotients. Hence these methods often only allow one to deduce
weaker statements concerning reduced quotients.
In Wiles' original arguments, however, the auxiliary modules~$M_D$ are free over the corresponding
Hecke algebras, and one ultimately deduces that the patched module~$M_{\infty}$  is also
free over~$R_{\infty}$,
from which one can certainly conclude that~$M_{\infty}/\aa$ is free over~$R_{\infty}/\aa$,
and hence that~$R = \T$. In our argument, we exploit the fact that,
by using all the Hecke operators, 
 the multiplicity one theorem for~$q$-expansions  allows us to also
 show that the auxiliary modules~$M_D$ are free 
over certain Hecke algebras,  and hence we are able to deduce (as in Wiles)
an integral~$R = \T$ theorem.

\begin{remark} \emph{{\bf An apology concerning notation:\rm} The notation~$Q$ that we have
used is meant to suggest a collection of Taylor--Wiles primes. Indeed, 
 the primes denoted by~$Q$ will play the role of Taylor--Wiles primes in the modularity
proof of~\S\ref{section:one}. However, in the proof
of Theorem~\ref{theorem:minimal} below, the set of primes~$Q$  will be fixed,
and there will be an auxiliary choice~$\QQ_D$ of Taylor--Wiles primes~$x \equiv 1 \mod p^D$. Explicitly,
we are proving an~$R = \T$ theorem at level~$p$ and a level which already
includes a fixed collection of Taylor--Wiles primes~$Q$. Hence we require a second auxiliary choice
of Taylor--Wiles primes for which we use the letter~$\QQ$ rather than~$Q$.}
\end{remark}

\begin{proof}
We first define a classical unmodified (``natural'') deformation ring $\Rn_{Q}$ which records deformations
which are of the same type as considered in $R_{Q}$, except  now the extra choice of
eigenvalues is omitted, 
as is the choice of eigenvalue at $\ell = p$. 
There is a natural isomorphism
$$R_{Q} \simeq \Rn_{Q} \otimes_{\Rloc} \Rlocmod,$$
where
$\displaystyle{\Rloc:=\widehat{\bigotimes}_{\ell | p \NQ} R_{\ell}}$
denotes the corresponding local deformation rings for $\Rloc$, and 
$$\Rlocmod:=\widehat{\bigotimes}_{\ell | p \NQ} \Rmod_{\ell}.$$
We remind the reader that one should think about the $\Rloc$-algebra~$\Rmod$ as follows: it 
is the algebra obtained by including the extra information over~$\Rloc$ coming from
a choice of Frobenius eigenvalue, and taking a localization of this ring corresponding
to fixing (residually) a choice of such an eigenvalue. In particular, the set of
components of the generic fibre
of~$\Rmod$ is a subset of the components of the generic fibre of~$\Rloc$.
The ring~$\Rmod$ is also reduced by Lemmas~\ref{prop:CM} and~\ref{lemma:kisin}.

We now patch together
coherent cohomology modules  and we also simultaneously
patch Betti cohomology.  
Namely,  we patch the pairs of modules 
$$M^{\co}_{D} = 
H^0(X_{H_D}(\NQ \cdot \QQ_D),\omega^{\otimes p}_{E/\OL})^{\vee}_{\wm},$$
$$M^{\be}_{D} = H^1(X_{H_D}(\NQ \cdot \QQ_D),\Sym^{p-2}((E/\OL)^2))^{\vee}_{\m}.$$
The notation $\co$ and $\be$ refers
 to  coherent and Betti cohomology, respectively.
Here~$\QQ_D$ is a collection of Taylor--Wiles primes~$x$ (distinct from primes
dividing~$\NQ$)  such that~$x \equiv 1 \mod p^D$,
and~$H_D$ is the subgroup of~$(\Z/\NQ \cdot \QQ_D \Z)^{\times}$ generated
by  the kernel of the
map~$(\Z/\QQ_D\Z)^{\times} \rightarrow (\Z/p^D \Z)^{\# \QQ_D }$ together with a fixed
subgroup of~$(\Z/\NQ \Z)^{\times}$.
The first module has a faithful action of~$\Tw_{Q \cdot \QQ_D,\wm}$, and the second
has a faithful action of~$\T_{Q \cdot \QQ_D,\m}$. Moreover, the tensor product
$M^{\be}_{D} \otimes_{\Rloc} \Rmod$ has a natural action of~$\Tw_{Q \cdot \QQ_D,\wm}$.
We patch together both of these modules for the following reason. The patched Betti cohomology
module is known to be nearly faithful over the patched framed natural Galois deformation
rings  $\Rloc\llbracket  x_1,\ldots,x_{r+d-1}\rrbracket$ by a theorem of Kisin~\cite{kisin}  --- this essentially
amounts to the fact that we already \emph{have} modularity lifting theorems in this context; the goal
is to upgrade these theorems to integral statements.
On the other hand, the coherent cohomology will be free over the corresponding modified
Hecke rings, which allows for an easier passage from patched objects back to finite level.

By Lemma~\ref{lemma:count}, the modules $M^{\co}_{D}$ are free of rank one over~$\Tw_{Q \cdot \QQ_D,\wm}$,
and so $M^{\co}_{D}/\varpi^D$ is free of rank one over~$\Tw_{Q \cdot \QQ_D,\wm}/\varpi^D$. On the other hand,
 $M^{\be}_{D}$ need not be free. However, the action of  the Hecke on 
 $M^{\be}_{D}/\varpi^D$ certainly factors through~$\T_{Q \cdot \QQ_D,\m}/\varpi^D$, and the action of
 the full Hecke algebra (with Hecke operators for primes dividing~$N$)  on $M^{\be}_{D}/\varpi^D \otimes_{\Rloc} \Rmod$ factors 
 through~$\Tw_{Q \cdot \QQ_D,\wm}/\varpi^D$.
 
 If we patch together all this data simultaneously (together with framings), we deduce (as in the
 pre-Diamond argument of Wiles) that the patched framed module $M^{\square,\co}_{\infty}$
is free of rank one over~$\Tw^{\square}_{\infty}$
(this is tautological, because each of the modules that is being patched will be free), 
and that the action
of the patched modified Galois deformation ring on $M^{\square,\be}_{\infty} \otimes_{\Rloc} \Rmod$ acts 
through~$\Tw^{\square}_{\infty}$.
If $\Rbox_{Q}$ denotes the framed version of $R_{Q}$ over $d$ primes dividing $p \NQ$, then the Taylor--Wiles  method as modified by Kisin
gives  presentations:
$$R_{Q}\llbracket  T_1,\ldots,T_{4d-1}\rrbracket \simeq
\Rbox_{Q} \simeq \Rmod\llbracket  x_1,\ldots,x_{r+d-1}\rrbracket/(f_1,\ldots,f_{r+1}),$$
$$\Rn_{Q}\llbracket   T_1,\ldots,T_{4d-1}\rrbracket \simeq
\Rboxn_{Q} \simeq \Rloc\llbracket  x_1,\ldots,x_{r+d-1}\rrbracket/(f_1,\ldots,f_{r+1}).$$
From~\cite{kisin} we know that~$M^{\square,\be}_{\infty}$ is a nearly faithful
$\Rloc\llbracket  x_1,\ldots,x_{r+d-1}\rrbracket$-module. From the freeness of~$M^{\co}$
over the Hecke algebra, it follows that~$\Tw^{\square}_{\infty}$ is a quotient
of the power series ring~$\Rmod \llbracket x_1,\ldots, x_{r+d-1} \rrbracket$.
 If~$\Tw^{\square}_{\infty}$
 is actually isomorphic to this ring,
then by taking the quotient by diamond operators, we arrive at the required
 isomorphism~$R_{Q} \simeq \Tw_{Q,\wm}$.
 On the other hand, the components of the generic fibre of $\Rmod \llbracket  x_1,\ldots,
 x_{r+d-1}\rrbracket$ are a subset of those for~$\Rloc$, so we deduce (cf. Lemma~2.3 of~\cite{Taylor}) that~$\Tw^{\square}_{\infty}$
 is a nearly faithful~$\Rmod \llbracket  x_1,\ldots,
 x_{r+d-1}\rrbracket$-module. Since~$\Rmod$ is reduced and Noetherian, the power series
 ring has no nilpotent elements, and hence being nearly faithful over this ring is equivalent
 to being faithful. Thus~$\Tw^{\square}_{\infty}$ is a faithful module, and
 hence isomorphic to~$\Rmod \llbracket x_1,\ldots, x_{r+d-1} \rrbracket$, and the
 proof follows.
\end{proof}

\begin{corr} \label{corr:metro} There is an isomorphism $R_Q/\varpi \simeq \Tw_{Q,\wm}/\varpi \simeq  \Tw_{Q,k,\wm}$.
\end{corr}

\begin{proof} The first isomorphism is an immediate consequence
of the previous theorem, so it suffices to show that $\Tw_{Q,\wm}/\varpi \simeq \Tw_{Q,k,\wm}$, or equivalently,
that $\Tw_{Q,\wm}/\varpi$ acts faithfully on $H^0(X,\omega^{\expp}_k)$. For this it suffices to
note that 
$$H^0(X,\omega^{\expp}_k)^{\vee} = (H^0(X,\omega^{\expp}_{E/\OL})[\varpi])^{\vee}
= H^0(X,\omega^{\expp}_{E/\OL})^{\vee}/\varpi$$
 and that $H^0(X,\omega^{\expp}_{E/\OL})^{\vee}$ is free of rank
one over~$\Tw_{\m}$ by Lemma~\ref{lemma:count} below.
\end{proof}

\section{Katz modular forms}
We now study more closely the action of Hecke operators in characteristic~$p$, especially in weights~$p$
and~$1$. In this section, we denote $\T_{Q,k}$ and $\Tw_{Q,k}$ by $\T$ and $\Tw$.
We use freely the $q$-expansion principle, namely, that a form 
in $H^0(X,\omega^n_{k})$ is determined by its image in $k\llbracket  q\rrbracket$.
Multiplication by the Hasse invariant induces a map:
$$A: H^0(X,\omega_k) \rightarrow H^0(X,\omega^{\expp}_k)$$
which is an injection and is the identity map on $q$-expansions. It follows that this map is $\T$-equivariant,
but it is not in general $\Tw$-equivariant. There is another map between these spaces induced by the
map $q \mapsto q^{\expp}$:
$$V: H^0(X,\omega_k) \rightarrow H^0(X,\omega^{\expp}_k),$$
this map $V$ is also $\T$-equivariant. (Although the corresponding Hecke algebras $\T$ in weights one and $p$ are
not the same rings, the meaning of $\T$-equivariance should be clear.)

\begin{lemma} \label{lemma:count} Let $\wm$ be a maximal ideal in $\Tw$ in weight~$p$, and let $\m$ be the corresponding
maximal ideal in $\T$. Then
$\dim H^0(X,\omega^{\expp}_k)[\wm] = 1$ and
$$\dim H^0(X,\omega^{\expp}_k)[\m] = 1 +  \dim H^0(X,\omega_k)[\m],$$
 where
 $\dim H^0(X,\omega_k)[\m] \le 1$.
\end{lemma}

\begin{proof} The Hecke operators at all primes determine the $q$-expansion, which proves the first equality.
For the second, let $f \in H^0(X,\omega^p_k)[\wm]$ denote the unique eigenform with leading coefficient~$1$.
There is a homomorphism
$$H^0(X,\omega^p_k)[\m] \rightarrow k\llbracket  q^{\expp}\rrbracket \cap H^0(X,\omega^{\expp}_k)[\m]$$
given by  $g \mapsto g - a_1(g) f$. By the $q$-expansion principle, the kernel of
this map is one-dimensional. 
By the main  theorem of Katz~\cite{Katz}, the
space $k\llbracket  q^{\expp}\rrbracket \cap H^0(X,\omega^{\expp}_k)[\m]$, which lies in the kernel $\ker(\theta)$ of the theta operator,
may be identified with the image of $H^0(X,\omega_k)[\m]$ under $V$.
 This gives the first
equality. To prove the  inequality, we repeat
the same argument in weight one, except now~\cite{Katz} implies that $ k\llbracket  q^{\expp}\rrbracket \cap H^0(X,\omega_k) = 0$.
\end{proof}

It follows that $\T^1 = \Tw^1$ in weight~one because $T_p \in \T^1$ and $\m \T^1 = {\wm} \T^1$.
\subsection{Doubled modules}
We define the notion of a doubled module with respect to $\T$ and~$\Tw$.

\begin{df} Let $N \subset H^0(X,\omega^{\expp}_k)$ be invariant under the action of $\Tw$,
let~$\TI = \Ann_{\Tw}(N)$, 
and let $I = \Ann_{\T}(N) = \TI \cap \T$. We say that $N$ is doubled if the action of $\Tw$ on $N$ acts
faithfully through a quotient $\Tw/{\TI}$ such that
$$\length(\Tw/{\TI}) = 2  \cdot \length(\T/I).$$
\end{df}

\begin{lemma} There exists a maximal doubled sub-module of $H^0(X,\omega^{\expp}_k)_{\m}$.
\end{lemma}

\begin{proof} 
If $\dim H^0(X,\omega^{\expp}_k)[\m] = 1$, then $H^0(X,\omega^{\expp}_k)^{\vee}_{\m}$
is free of rank one over $\T_{\m}$ and $\Tw_{\m}$, so $\T_{\m} \simeq \Tw_{\m}$ in that case,
and the maximal doubled quotient is trivial. Hence we assume that
$\dim H^0(X,\omega^{\expp}_k)[\m] = 2$. By Nakayama's Lemma applied to $\T_{\m}$, it follows that 
$\Tw_{\m}$ has rank at most $2$ over $\T_{\m}$, or equivalently that $T_p$ satisfies a quadratic
relation. If $N$ is doubled, however,
then~$\Tw_{\m}/\TI$ must  be free of rank two as a~$\T_{\m}/I$-module.
In particular,
 $I$ must act trivially on  $\Tw_{\m}/\T_{\m}$,
so it must contain the annihilator of this module. Let $J$ be the
annihilator of $\Tw_{\m}/\T_{\m}$ as a~$\T_{\m}$-module. This is an ideal of~$\T_{\m}$;
we claim that it is actually an ideal of~$\Tw_{\m}$. 
By definition, if~$a \in \T_{\m}$ is any element, then~$a$ annihilates~$\Tw_{\m}/\T_{\m}$
if and only if it lies in~$J$.
Equivalently,  we have $a x \in \T_{\m}$ for all~$x \in \Tw_{\m}$ if and only if~$a \in J$.
To show that~$J$ is an ideal of~$\Tw_{\m}$, it suffices to show
that $a T_p \in J$.  By the previous equivalence, we have~$a T_p \in \T_{\m}$.
Moreover, since~$a x \in \T_{\m}$ for every~$x \in \Tw_{\m}$, we also have~$a T_p x \in \T_{\m}$
for every~$x \in \Tw_{\m}$. Thus~$a T_p \in J$, and~$J$ is an ideal of~$\Tw_{\m}$.
We then observe that
  $H^0(X,\omega^{\expp}_k)[J]$ is doubled, and is thus the maximal doubled
  sub-module.
\end{proof}

The ideal~$J$ is the analogue in this context of the (global) doubling ideal
denoted$\JJg$ in~\cite{CG}.

Let $M \subset H^0(X_{k},\omega^{\expp})_{\m}$ be a maximally doubled module. 
Hence $M^{\vee}$ is free of rank $2$ over ${\Tp}/J$ and free of rank~$1$ over $\Tw/J$,
where $J = \Ann_{{\Tp}}(M)$.
The only maximal ideal of ${\Tp}/J$ is $\m$, so ${\Tp}/J$ is a finite local ring.
Let $\ker(\theta)$  denote the subset of elements annihilated by the $\theta$ operator,
and let $\ker_M(\theta) = \ker(\theta) \cap M$

\begin{lemma} \label{lemma:faithful} The module $M/\ker_M(\theta)$ is a faithful ${\Tp}/J$-module. 
\end{lemma}

\begin{proof} We have a surjection:
$$({\Tp}/J)^2 \simeq M^{\vee} \rightarrow \ker_M(\theta)^{\vee}$$
The module $\ker_M(\theta)$ is isomorphic as a $\Tp$-module to $H^0(X,\omega_k)_{\m}$. Hence,
by Lemma~\ref{lemma:count}, $\ker_M(\theta)^{\vee}$ is cyclic as a $\Tp/J$-module. 
 If $K$ denotes the kernel, it follows that $K/\m \rightarrow ({\Tp}/\m)^2$ has non-trivial image.
Let $x \in K$ denote an element which maps to a non-zero element in $({\Tp}/\m)$. Then the cyclic
module in $({\Tp}/J)^2$ generated by $x$ is a faithful ${\Tp}/J$-module, and hence $K$ is also a faithful
${\Tp}/J$-module. We then have $M/\ker_M(\theta) = K^{\vee}$.
\end{proof}

\begin{df} There is a $\Tp$-equivariant  pairing 
$\Tp/J \times M \rightarrow k$ defined  as follows:
$$\langle T_{n}, f \rangle = a_{1}(T_n f).$$
\end{df}

\begin{lemma} \label{lemma:perfect} The map $\langle *,* \rangle$  
 is a perfect pairing between $\Tp/J$ and $M/\ker_M(\theta)$.
\end{lemma}

\begin{proof} If $f = \sum a_n q^n$ and $\theta(f) = 0$, then $a_n = 0$ for all $(n,p) = 1$, so 
$\langle T_n,f \rangle = 0$ for all $T_n \in \Tp$, and hence for all $T \in \Tp$. Conversely, if $\langle T_n ,f \rangle = 0$, then $a_n = 0$
for all $(n,p) = 1$ and $f$ lies in the kernel of $\theta$. Now suppose that $\langle T,f \rangle = 0$ for all $f \in M$.
Since the map is Hecke equivariant, it follows that
$$\langle T_n, T \kern-.02em{f} \rangle = \langle T \kern+0.1em{T_n}, f \rangle =  \langle T_n T, f \rangle  =  \langle T, T_n f\rangle = 0$$
for all $n$, and hence $T \kern-.02em{f}$ is trivial in $M/\ker_M(\theta)$. But $\Tp/J$ acts faithfully on $M/\ker_M(\theta)$, so $T = 0$.
\end{proof}

 \begin{lemma} $\ker_M(\theta)$ is a faithful $\Tp/J$-module, and $\ker_M(\theta)^{\vee}$ is free over $\Tp/J$ of rank one.
 \end{lemma}
 
 \begin{proof} By definition, $M$ is free of rank two over $\Tp/J$, and so it has the same
 length as $(\Tp/J)^2$. However, by Lemma~\ref{lemma:perfect},
 $M/\ker_M(\theta)$ and $\Tp/J$ have the same dimension over $k$, and hence the same length.
 It follows that $\ker_M(\theta)$ has the same length as $\Tp/J$. Since $\dim \ker_M(\theta)[\m] = 1$,
 the module $\ker_M(\theta)^{\vee}$ is cyclic of the same length as $\Tp/J$, and thus free of rank one
 over $\Tp/J$. Hence $\ker_M(\theta)$ is also faithful as a $\Tp/J$-module.
\end{proof} 
 
Since, by Lemma~\ref{lemma:count}, $\ker(\theta)^{\vee}_{\m}$ is also free of rank one over $\Tone = \Tone_{Q,k}$,
the Hecke algebra in weight one, and 
$\ker_M(\theta)^{\vee}$ is a quotient of this module,  we deduce the
immediate corollary:

\begin{corr} \label{cor:pickle}There is a surjection $\Tone \rightarrow \T/J$.
\end{corr}

Now let us fix~$X = X_H(\NQ)$, and suppose that~$\m$  and~$\wm$
correspond to our residual Galois representation~$\rhobar$ together with a suitable choice of~$a_{\ell}$.

\begin{prop} \label{prop:porpoise} There exists a doubled submodule $M \subset  H^0(X_H(\NQ),\omega^{\expp}_k)_{\m}$ such that the action of
$\T$ on $M$ acts faithfully through $R^1_Q/\varpi$. In particular, there is a surjection~$\T/J \rightarrow R^1_Q/\varpi$.
\end{prop}

\begin{proof}
Let~$\RR^1_Q$ denote the modification of~$R_Q$ where one also takes into account an eigenvalue~$\alpha_p$
of~$\rho(\Frob_p)$ (recall that representations associated to~$D^1_Q$ are unramified at~$p$). 
The ring~$\RR^1_Q$ is a finite flat degree~$2$ extension of~$R^1_Q$ given as the quotient of~$R^1_Q[\alpha_p]$
by the monic quadratic polynomial~$\alpha_p$ corresponding to  the characteristic polynomial
of Frobenius in the universal representation associated to~$R^1_Q$.
Let us distinguish two cases. The first is when the eigenvalues of~$\rhobar(\Frob_p)$ are distinct, and the
second is when they are the same (in the latter case,~$\rhobar(\Frob_p)$ may or may not be scalar).
Let~$\Sigma$ denote the set of eigenvalues, so~$|\Sigma| = 2$ or~$1$.
If~$|\Sigma| = 1$, then~$\RR^1_Q$ is a local ring, and if~$|\Sigma| = 2$, it is a semi-local ring with two maximal
ideals; indeed, by Hensel's Lemma the quadratic relation satisfied by~$\alpha_p$ splits over~$R^1_Q$, and so there is an isomorphism~$\RR^1_Q = R^1_Q \oplus R^1_Q$ in this case.
(This is essentially Lemma~\ref{lemma:hensel}.)
In particular, the quadratic polynomial has exactly two roots in~$\RR^1_Q$.
There is a surjection
$$\bigoplus_{\Sigma} R_Q \rightarrow \bigoplus_{\Sigma} R_Q/\varpi \rightarrow \RR^1_Q/\varpi.$$
Here the sum is over the rings~$R_Q$ corresponding to each choice of eigenvalue~$a_p \in \Sigma$.
The latter map sends~$\alpha_{\ell}$ to~$\alpha_{\ell}$ for all~$\ell$ dividing~$ \NQ$. If~$|\Sigma| = 2$,
then each~$\alpha_p$ goes to the corresponding eigenvalue of Frobenius.
If~$|\Sigma| = 1$, then~$\alpha_p$ goes to~$\alpha_p$. These maps are well defined because,
after reduction modulo~$\varpi$, all the local conditions (including the determinant)
in the definition of~$D^1$ and~$D^1_Q$ coincide with the exception of~$\ell = p$. For~$\ell = p$,
the enriched ring~$\RR^1_Q$ recieves a map from~$\Rmod_p$, because~$\alpha_p$ in~$\Rmod_p$
is exactly an eigenvalue of Frobenius. 
Since $\Tw_{Q,k,\wm} \simeq R_Q/\varpi$ by Corollary~\ref{corr:metro}, there is a surjection
$$\bigoplus_{\Sigma} \Tw_{Q,k,\wm} \rightarrow \RR^1_Q/\varpi.$$
Since $\bigoplus H^0(X_H(\NQ),\omega^{\expp}_k)_{\wm}$
is co-free over $\bigoplus \Tw_{Q,k,\wm}$, there certainly exists a module $M$ 
such that the action of~$\Tw$ on~$M$ is precisely via this quotient~$\RR^1_{Q}/\varpi$.
Yet this quotient is also finite flat of degree two over~$R^1_Q/\varpi$, which is precisely the image of~$\T$ under this map.
Hence  the submodule~$M$ is doubled, and the corresponding action of~$\T$ is via~$R^1_Q/\varpi$.
The final claim follows from the fact that~$J$ is the ideal corresponding to the largest doubled submodule.
\end{proof}

\begin{remark} \emph{One alternative way to write this paper was to define the functors~$D_Q$, etc. \emph{without}
making a fixed choice of~$a_{\ell}$. This would have amounted to replacing the universal
local deformation rings~$R_Q$, etc. by universal semi-local deformation rings, which would be 
isomorphic to a direct sum
over all the local rings in this paper and over all possible choices of~$a_{\ell}$. We have decided to work with
the version of these rings in which choices have been made, however, as evidenced by the proof
of the previous proposition, one
still has to deal with semi-local rings in some cases at~$\ell = p$, because when~$\rhobar(\Frob_p)$ has distinct eigenvalues,
the corresponding maximal ideal in weight one is determined by the sum~$\alpha + \beta$ of these eigenvalues whereas
the local rings in higher weight require a choice of~$\alpha$ or~$\beta$. }
\end{remark}

\section{Passage from weight~$p$ to weight~one}

Let $\Tone_{Q} = \Tone_{Q,E/\OL}$. The rings ~$\T^1$, $\Tw^1$, $\T$, and~$\Tw$ of the previous section were
abbreviations for
the rings~$\T^1_{Q,k}$, $\Tw^1_{Q,k}$, $\T_{Q,k}$ and~$\Tw_{Q,k}$ respectively; we return to this expanded notation now.

\begin{corr}  \label{corr:nail} There is an isomorphism $R^1_{Q}/\varpi \rightarrow \Tone_{Q,k,\m} \simeq \Tone_{Q,\m}/\varpi$.
\end{corr}

\begin{proof} For the first isomorphism, it suffices to note that there exists a map
$$R^1_{Q}/\varpi \rightarrow \Tone_{Q,k,\m} \rightarrow \T_{Q,k,\m}/J \rightarrow R^1_{Q}/\varpi$$
whose composite is the identity.
The existence of the first map follows from the fact that Galois representations in weight~one are unramified
at~$p$, which follows from Theorem~3.11 of~\cite{CG} (together with the appropriate local--global
compatibility away from~$p$, which follows as in the proof of Theorem~3.11 of~\cite{CG} by reduction to characteristic zero in higher weight, together with the proof of Proposition~\ref{prop:action}). The second map comes from Corollary~\ref{cor:pickle}. The
existence of the
third map follows from Proposition~\ref{prop:porpoise}. The identification of $\Tone_{Q,k,\m}$
with $\Tone_{Q,\m}/\varpi$ follows from the fact that
$$H^0(X_H(\NQ),\omega_k)_{\m} = H^0(X_H(\NQ),\omega_{E/\OL})_{\m}[\varpi]$$
and the fact that $H^0(X_H(\NQ),\omega_{E/\OL})^{\vee}$ is free over $\Tone_{Q,\m}$ by Lemma~\ref{lemma:count}.
\end{proof}

\subsection{Modularity in weight one} \label{section:one}
Now that we have an isomorphism $R^1_{Q}/\varpi \simeq \T^1_{Q,\m}/\varpi$ for all collections of Taylor--Wiles primes $Q = Q_D$,
and we apply the machinery of~\cite{CG}, in particular Prop.~2.3 as applied in~\S3.8 of~\emph{ibid}.
We patch the modules 
$H^0(X_{H_D}(\NQ_D),\omega_{E/\OL})^{\vee}_{\m}$,
where~$Q$ is a collection of Taylor--Wiles primes~$x \equiv 1 \mod p^D$, and~$H$
is the subgroup of~$(\Z/\NQ_D \Z)^{\times}$ which is generated by the kernel 
of~$(\Z/ Q_D \Z)^{\times} \rightarrow (\Z/p^D \Z)^{\times}$ and the~$p$-Sylow of~$(\Z/N\Z)^{\times}$.
We obtain 
a module $M_{\infty}$ which is a module over
the framed ring
of diamond operators $S^{\square}_{\infty}$, 
and a patched deformation ring $\Rboxone_{\infty}$ which is also an algebra
over this ring.  In contrast to~\cite{CG}, the ring~$\Rboxone_{\infty}$ is a power series ring over a completed
tensor product
 $$\Rlocmod:=\widehat{\bigotimes}_{\ell | N} \Rmod_{\ell},$$
 instead of a power series ring over~$\OL$.
Because
the modules $H^0(X_{Q},\omega_{E/\OL})^{\vee}_{\m}$ are free over $\T_{Q,\m}$, the module $M_{\infty}$ is
cyclic as an $\Rboxone_{\infty}$-module. 
Hence we know that:
\begin{enumerate}
\item $M_{\infty}/\varpi$ is  free of rank one over $\Rboxone_{\infty}/\varpi$, as follows
from our mod-$p$ modularity results above, in particular Corollary~\ref{corr:nail}.
\item $M_{\infty}$  is pure of co-dimension one as an $S^{\square}_{\infty}$-module; that is, 
$M_{\infty}$ is a torsion $S^{\square}_{\infty}$-module, and there
exists a presentation:
$$0 \rightarrow (S^{\square}_{\infty})^n \rightarrow (S^{\square}_{\infty})^n \rightarrow M_{\infty} \rightarrow 0;$$
this is exactly the output of the construction of~\cite{CG}.
\end{enumerate}

The second result is essentially a formal consequence of the method of~\cite{CG} rather than anything in this paper.
This on its own is enough to show that~$M_{\infty}$ will certainly be supported on~\emph{some} components 
of the generic fibre of~$\Rboxone_{\infty}$. However, as soon as~$S$ contains primes for which~$\rhobar$ is unramified
(that is, as soon as we are at non-minimal level), 
the ring~$\Rboxone_{\infty}[1/\varpi]$ will have multiple components. The usual technique for showing that the support of~$M_{\infty}$
is spread over all components is to produce modular lifts with these properties. In our context this is not
possible: there are no  weight one forms in characteristic zero which are Steinberg at a finite place~$q$
(see the proof of Prop.~\ref{prop:katzfail} below).
Our replacement for producing modular points in characteristic zero is to work 
on the special fibre, and to show that~$M_{\infty}/\varpi$
is (in some sense) spread out as much as possible over~$\Rboxone_{\infty}/\varpi$. And we do this (and this is the main point
of everything so far) by working in weight~$p$ and then descending back to weight one using the doubling method.
In particular, we know that~$M_{\infty}/\varpi$  is  free of rank one over $\Rboxone_{\infty}/\varpi$.
From these two properties, we will now deduce that $M_{\infty}$ is free of rank one over $\Rboxone_{\infty}$,
which will imply all our modularity results.

By Nakayama's Lemma, there is certainly an exact sequence of $S^{\square}_{\infty}$-modules.
$$0 \rightarrow K \rightarrow   \Rboxone_{\infty} \rightarrow M_{\infty} \rightarrow 0$$
It suffices to show that~$K = 0$.  By Nakayama's Lemma again, it suffices to show that $K/\varpi = 0$.
Tensoring with $S^{\square}_{\infty}/\varpi$ (that is, reducing modulo~$\varpi$), we get a long exact sequence
$$\Tor^1(S^{\square}_{\infty}/\varpi,M_{\infty}) \rightarrow K/\varpi
\rightarrow \Rboxone_{\infty}/\varpi \stackrel{\simeq}{\longrightarrow} M_{\infty}/\varpi \rightarrow 0.$$
Here the last map is an isomorphism by property~$(1)$ above.  Hence, to prove that~$K/\varpi$ is trivial, it suffices
to show that
$$\Tor^1(S^{\square}_{\infty}/\varpi,M_{\infty}) = M_{\infty}[\varpi]$$
is trivial. If $M_{\infty}[\varpi]$ is non-trivial, then, from the purity of $M_{\infty}$,  we claim that  $M_{\infty}/\varpi$ will have positive rank over
$S^{\square}_{\infty}/\varpi$. To see this,  simply tensor the presentation of $M_{\infty}$ with
$S^{\square}_{\infty}/\varpi$ to obtain the exact sequence:
$$0 \rightarrow \Tor^1(S^{\square}_{\infty}/\varpi,M_{\infty}) \rightarrow
(S^{\square}_{\infty}/\varpi)^n \rightarrow (S^{\square}_{\infty}/\varpi)^n \rightarrow M_{\infty}/\varpi \rightarrow 0,$$
from which it follows that
$$\rank_{S^{\square}_{\infty}/\varpi} M_{\infty}/\varpi =\rank_{S^{\square}_{\infty}/\varpi} 
\Tor^1(S^{\square}_{\infty}/\varpi,M_{\infty}).$$
If~$\Tor^1(S^{\square}_{\infty}/\varpi,M_{\infty})$ is non-zero, then, as it is a submodule of a free module over
$S^{\square}_{\infty}/\varpi$ which has no annihilator, it also has no annihilator
 as an~$S^{\square}_{\infty}/\varpi$ module. However,
a  module with no annihilator over a power series ring over~$k$ certainly must have positive rank.
This implies that (if~$M_{\infty}[\varpi]$ is non-zero) that~$M_{\infty}/\varpi$ has positive rank.
Yet this
 contradicts  the fact that $\Rboxone_{\infty}/\varpi \simeq M_{\infty}/\varpi$ does not  have positive rank, as $\Rboxone_{\infty}$
is flat over $\OL$ (Lemma~\ref{lemma:kisin}) and of smaller dimension than the ring of diamond operators (by one).
Hence $M_{\infty}$ is free of rank one over $\Rboxone_{\infty}$. But now specializing down to finite level,
we deduce that $H^0(X,\omega_{E/\OL})^{\vee}_{\m}$ is free of rank one over $R^1$, which
proves Theorem~\ref{theorem:main}.

\subsection{Producing Torsion Classes} \label{section:blog}
Let 
$$f = \sum a_n q^n \in S_1(\Gamma_H(N),\eta)$$
be a cuspidal eigenform of some  level~$N$ and character~$\eta$. Let
$$\rho: G_{\Q} \rightarrow \GL_2(E)$$
denote the corresponding Artin representation.  Assume
that~$\rho|D_{\ell}$ is reducible for any prime~$\ell$ dividing~$N$.
\begin{prop} \label{prop:katzfail} Let~$f$ be as above. Let~$p>2$ be a prime such that~$\rhobar_f$ is absolutely
irreducible and~$p$ is prime to the level~$N$ of~$f$ and the order of~$\eta$.
Then there exists  a set of primes~$\ell$ of positive
density so that, for each such~$\ell$, the map:
$$H^0(X( \Gamma_H(N) \cap \Gamma_0(\ell))_{\Z_p},\omega)_{\m} \otimes \F_p \rightarrow H^0(X( \Gamma_H(N) \cap \Gamma_0(\ell))_{\F_p},\omega)_{\m}$$
is not surjective. \end{prop}

\begin{remark} \emph{This implies that Katz' base-change theorem
(Theorem~1.7.1 of~\cite{KatzBig}) fails as badly as possible in weight one.}
\end{remark}

\begin{proof} Suppose, to the contrary, that all
such forms of this level lift to characteristic zero.
There are no 
forms in characteristic zero which are new at~$\ell$ of level $\Gamma_0(\ell)$, because any such form would have to be (up to unramified twist)
 Steinberg  at~$\ell$,
and no weight one form in characteristic zero can be Steinberg at any place.
The easiest way to see this is that the eigenvalue of~$U_{\ell}$ would have to be a root of unity times
${\ell}^{-1/2}$, but this is impossible because Hecke eigenvalues
of  modular forms are algebraic integers.
Hence any Galois representation arising
from forms of this level must come from level $\Gamma_H(N)$, and so in particular be unramified at~${\ell}$.
Thus, by Theorem~\ref{theorem:main}, it suffices to show
that there is a non-trivial deformation of~$\rhobar$ to
the dual numbers which is minimal at~$N$,  corresponds to
a quotient of~$\Rmod_{{\ell}}$ at the new auxiliary prime~${\ell}$, and is
 unramified everywhere else. 
The reduced tangent space of the unramified
deformation ring is given by the
Selmer group $H^1_{\emptyset}(\Q,\ad^0(\rhobar))$.
Denote the dual Selmer group by $H^1_{\emptyset^*}(\Q,\ad^0(\rhobar)(1))$.
Since these groups are both finite, there exists a finite extension
$F/\Q$ which contains the fixed field of $\ker(\rho)$ and such that
all the classes in  $H^1_{\emptyset^*}(\Q,\ad^0(\rhobar)(1))$ split
completely. Let~${\ell}$ be a prime which splits completely in~$F(\zeta_p)$.
Let $H^1_{\Sigma}(\Q,\ad^0(\rhobar))$ denote the modified Selmer group
where classes are allowed to be arbitrarily ramified at~${\ell}$. By construction of~${\ell}$, the dual Selmer group $H^1_{\Sigma^*}(\Q,\ad^0(\rhobar))$
consisting of all dual Selmer classes which split completely at~${\ell}$ is equal
to $H^1_{\emptyset^*}(\Q,\ad^0(\rhobar))$, because the localization map
factors through the restriction to $G_F$, and by construction all classes in the latter group are assumed to
split completely over~$\Q_{\ell}$. 
Hence,  the Greenberg--Wiles Euler characteristic formula:
$$\begin{aligned}
 \frac{|H^1_{\Sigma}(\Q,\ad^0(\rhobar))|}{|H^1_{\emptyset}(\Q,\ad^0(\rhobar))|} & \ = 
\frac{|H^1_{\Sigma}(\Q,\ad^0(\rhobar))|}{|H^1_{\Sigma^*}(\Q,\ad^0(\rhobar)(1))|} \cdot
\frac{|H^1_{\emptyset^*}(\Q,\ad^0(\rhobar)(1))|} {|H^1_{\emptyset}(\Q,\ad^0(\rhobar))|}\\
& \ =\frac{|H^1(\Q_{\ell},\ad^0(\rhobar))|}{|H^1(\F_{\ell},
\ad^0(\rhobar))|}  = |H^0(\Q_{\ell},\ad^0(\rhobar)(1))|  = |\ad^0(\rhobar)(1)|, \end{aligned}$$
the final equality coming from the assumption that~${\ell}$ splits completely in~$F(\zeta_p)$.
Note that $\dim \ad^0(\rhobar)(1) = 3 > 0$.
It follows
that for such choices of~$\ell$,
 there exists a deformation
$$\rho: G_{\Q} \rightarrow \GL_2(k[\eps]/\eps^2)$$
which is minimal at all primes away from~$\ell$ and genuinely
ramified at~$\ell$. 
Moreover, $\rhobar |D_{\ell}$ is trivial and $\ell \equiv 1 \mod p$. 
 It suffices to show that the corresponding deformation
arises from a quotient of~$\Rmod_{\ell}$,
which was described explicitly in this case
by Lemma~\ref{lemma:special}.
Since we are considering
fixed determinant deformations, the 
trace of the image~$\rho(g)$ of any element~$g$ is $1 + \det(g) = 2$.
However, it is apparent the description of~$\Rmod_{\ell}$ in
Lemma~\ref{lemma:special} that all the relations apart from $\tr(\rho(\tau)) = 2$ lie in~$\m^2$, and so
are automatically satisfied for any deformation to $k[\eps]/\eps^2$. (Note that, associated to~$\rho$,
there is a  corresponding
surjection~$\Rmod_{\ell} \rightarrow k[\eps]/\eps^2$ for any choice of~$\alpha_{\ell} \in 1 + \eps \cdot k[\eps]/\eps^2$.)
\end{proof}

\subsection{Proof of Theorem~\ref{theorem:blog}}
To prove Theorem~\ref{theorem:blog}, it suffices to apply Proposition~\ref{prop:katzfail}  to suitably
chosen~$f$. Note that the class numbers of the fields~$\Q(\sqrt{-23})$ and~$\Q(\sqrt{-47})$ are~$3$ and~$5$ respectively. This gives rise to suitable weight one forms~$f$ with image~$D_3$ and~$D_5$ and level~$\Gamma_1(23)$ and~$\Gamma_1(47)$  respectively (both with quadratic nebentypus). Applying Proposition~\ref{prop:katzfail}, we deduce the
existence of mod-$p$ Katz modular forms which fail to lift for all $p \ne 2,3,23$ in the first example and 
$p\ne 2,5,47$ in the second.  For~$p=2$, the theorem is known by an example of Mestre~\cite{Mestre}, 
completing the proof.

\begin{remark} \emph{The computations of~\cite{schaeffer} (see also~\cite{buzzcomp}) suggest that Theorem~\ref{theorem:blog}
should also be true if one insists that~$f$ is an \emph{eigenform}. By Serre's conjecture, this would follow
if for each~$p > 5$  there existed  an odd Galois representation:
$\rhobar: G_{\Q} \rightarrow \GL_2(k)$  unramified at~$p$  with image containing~$\SL_2(\F_p)$,
although proving this appears difficult.}
\end{remark}

\begin{remark} \emph{If the cuspform~$f$ in Proposition~\ref{prop:katzfail} is exceptional
 ---  that is, the projective image of~$\rho_f$ is~$A_4$,~$S_4$, or~$A_5$ --- then
the resulting torsion class at level $\Gamma_H(N) \cap \Gamma_0(q)$ will
not lift to characteristic zero at any higher level. The reason is that any form~$g$ with
$\rhobar_f = \rhobar_g$ will have to satisfy $\rho_f = \rho_g$ up to a~$p$-power twist, and the resulting
Hecke algebra in weight one cannot give rise to the infinitesimal deformations~$\rho$ which
arise in the proof of Proposition~\ref{prop:katzfail}.
}
\end{remark}

\bibliographystyle{amsalpha}
\bibliography{doubled}

 \end{document}